\documentclass[twocolumn]{ieeeconf}

\usepackage{graphics,epsfig}
\usepackage{tikz}
\usepackage{amsmath,amssymb,bm,mathtools}
\let\theoremstyle\relax

\usepackage{amsthm}
\DeclareMathOperator{\sign}{sign}
\theoremstyle{definition}
\newtheorem{definition}{Definition}
\newtheorem{theorem}{Theorem}
\newtheorem{lemma}{Lemma}
\newtheorem{assumption}{Assumption}

\newtheorem{problem}{Problem}

\newtheorem{remark}{Remark}
\newtheorem{property}{Property}

\newcommand{\interior}[1]{%
  {\kern0pt#1}^{\mathrm{o}}%
}

\newcommand{\boundary}[1] {
\partial#1
}

\usepackage{xcolor}
\usepackage{booktabs}

\usepackage{blindtext}

\begin{document}

\title{\LARGE\bf Optimal Control Barrier Functions: Maximizing the Action Space Subject to Control Bounds}
\author{Logan E. Beaver, \IEEEmembership{Member, IEEE}
\thanks{L.E. Beaver is with the Department of Mechanical \& Aerospace Engineering, Old Dominion University, Norfolk, VA 23529 USA (e-mail: lbeaver@odu.edu).}
}

\maketitle


\begin{abstract}
This letter addresses the constraint compatibility problem of control barrier functions (CBFs), which occurs when a safety-critical CBF requires a system to apply more control effort than it is capable of generating.
This inevitably leads to a safety violation, which transitions the system to an unsafe (and possibly dangerous) trajectory.
We resolve the constraint compatibility problem by constructing a control barrier function that maximizes the feasible action space for first and second-order constraints, and we prove that the optimal CBF encodes a dynamical motion primitive.
Furthermore, we show that this dynamical motion primitive contains an implicit model for the future trajectory for time-varying components of the system.
We validate our optimal CBF in simulation, and compare its behavior with a linear CBF.
\end{abstract}


\section{Introduction}

Control Barrier Functions (CBFs) have recently emerged as a mathematically rigorous method to guarantee constraint satisfaction for nonlinear systems with affine dynamics \cite{Ames2019ControlApplications,xiao2023safe}.
CBFs have been used as an additional "safety layer" to track a reference trajectory \cite{Xiao2021BridgingVehicles}, and to directly generate trajectories for reactive systems, e.g., in ecologically-inspired systems \cite{Egerstedt2018RobotAutonomy}.
CBFs have seen a wide range of applications, including connected and autonomous vehicles \cite{Xiao2021BridgingVehicles}, mobile robots \cite{Notomista2019TheRobot}, legged robots \cite{grandia2021multi}, and aerial vehicles \cite{singletary2022onboard}.

As CBFs are deployed on more realistic and complicated systems, new questions have arisen about the existence of feasible control actions that can jointly satisfy all of the constraints on the system.
This is also known as the \textit{constraint compatibility problem}, and it is a common issue in constrained control.
There have been many proposed solutions; imposing `compatibility constraints' to ensure constraints don't conflict \cite{Xu2022FeasibilityFunctions}, validating (or falsifying) whether constraints are compatible before deploying a system \cite{Tan20222CompatibilitySystems}, learning individual control barrier functions to maximize the feasible action space \cite{Xiao2023LearningCBFs}, selectively deactivating constraints \cite{lindemann2018control}, and the use of slack variables to soften constraints that are not safety-critical \cite{notomista2019optimal}.

In contrast, this letter considers a dynamical system with a single safety-critical constraint that is subject to control bounds.
This is the simplest system where the constraint incompatibility arises \cite{Xu2022FeasibilityFunctions}, and it presents an opportunity for us to analyze the fundamental properties of the optimal CBF in detail.
The contributions of this letter are as follows:
\begin{itemize}
    \item We present a constructive proof to derive the optimal CBF (Theorems \ref{thm:first-order} and \ref{thm:second-order}).
    \item We derive a switching system that behaves equivalently to the optimal CBF (Remark \ref{rmk:switching}).
    \item We derive the dynamical motion primitive the system follows when the optimal CBF is active (Definition \ref{def:line-int}).
    \item We show how the optimal CBF implicitly embeds a model of the environment over a significant time horizon (Remark \ref{rmk:embedded}).
\end{itemize}

The remainder of this article is organized as follows.
We briefly address notation in Section \ref{sec:notation} and present preliminary information about CBFs in Section \ref{sec:prelims}.
We address optimal first-order CBFs in Section \ref{sec:first-order}, and we present our main results for second-order systems in Section \ref{sec:second-order}.
Section \ref{sec:simulation} contains simulation results for an adaptive cruise control problem, and Section \ref{sec:conclusion} presents the conclusions and motivates areas of future work.

\subsection{Notation} \label{sec:notation}

Throughout our exposition we use standard notation from multiple domains.
First, we write vectors with bold letters and scalars as unstyled letters, i.e., $\bm{x}$ and $x$, respectively.
We write exogenous variables as explicit functions of time $\delta(t)$, and we omit the dependence of endogenous variables, i.e., states, on time.
We write the spatial derivative of a scalar function $b(\bm{x})$ along a vector field $f(\bm{x})$ using Lie derivative notation,$\frac{\partial f}{\partial \bm{x}} f(\bm{x}) \coloneqq L_{\bm{f}}\, g.$
We examine constraints of the form $b(\bm{x}) \leq 0$ rather than $b(\bm{x}) \geq 0$.
This means that the corresponding class $\mathcal{K}_{\infty}$ function is $\alpha(-b(\bm{x}))$.
Finally, we refer to the barriers in this article as \textit{Zeroing Barrier Functions} (ZBF) in general, and \textit{Control Barrier Functions} (CBFs) when the barrier is a class $\mathcal{K}_{\infty}$ function.

\section{Control Barrier Function Preliminaries} \label{sec:prelims}

We consider a nonlinear dynamical system with state $\bm{x}\in\mathbb{R}^m$, control action $\bm{u}\in\mathbb{R}^n$, and control-affine dynamics,
\begin{equation} \label{eq:dynamics}
    \dot{\bm{x}} = \bm{f}(\bm{x}) + \bm{g}(\bm{x}) \bm{u}.
\end{equation}
We consider the usual assumptions, namely, \eqref{eq:dynamics} is Lipschitz continuous, autonomous, and $\bm{f}, \bm{g}$ are continuous vector fields.
We seek trajectories that satisfy an inequality constraint,
\begin{equation} \label{eq:constraint}
    b(\bm{x}) \leq 0.
\end{equation}
The constraint has order $r\in\mathbb{N}$ if an only if\begin{equation} 
    \begin{cases}
        L_{\bm{g}}^k\,b\,\bm{u} = 0 \quad \text{ if } k < r, \\
        L_{\bm{g}}^k\,b\,\bm{u} \neq 0 \quad \text{ if } k = r,
    \end{cases}
\end{equation}
i.e., the control action $\bm{u}$ first appears after taking $r$ time derivatives of the constraint $b(\bm{x})$.
Finally, we we impose a maximum control effort bound,
\begin{equation} \label{eq:control-bounds}
    \bm{u}\in\mathcal{U} = \Big\{\bm{u}\in\mathbb{R}^m :~ ||\bm{u}|| \leq u_{\max}\Big\}.
\end{equation}

For a first order system, where $\bm{u}$ shows up explicitly in $\dot{b}$, it is common to employ a control barrier function (CBF) with the form \cite{Ames2019ControlApplications},
\begin{equation} \label{eq:cbf-form}
    \dot{b}(\bm{x}) = L_{\bm{f}}b + L_{\bm{g}}b\,\bm{u} \leq -\alpha(b(\bm{x})),
\end{equation}
where  $\alpha(-b(\bm{x}))$ is an extended class $\mathcal{K}_{\infty}$ function.
This is also a zeroing barrier function (ZBF) \cite{Ames2019ControlApplications} because $b(\bm{x}) = 0$ implies $\alpha(b(\bm{x})) = 0$, and thus $\dot{b}(\bm{x}) \leq 0$.
This satisfies the Nagumo theorem \cite{xiao2023safe}, which guarantees that the sub-level set $b(\bm{x})\leq 0$ is forward-invariant.

The challenge we address in this letter is the constraint compatibility problem, i.e., how to design a ZBF that simultaneously satisfies the constraint \eqref{eq:constraint} and control bounds \eqref{eq:control-bounds}.
For example, when $L_{\bm{g}}b > 0$, a scalar system must jointly satisfy,
\begin{equation}
\begin{cases}
    u \leq -\frac{\alpha(b) + L_{\bm{f}}b}{L_{\bm{g}}b}, \\
    u \geq -u_{\max}.
\end{cases}
\end{equation}
Thus, if the system ever enters a state satisfying
\begin{equation}
    -\alpha(b) < -u_{\max} L_{\bm{g}}b + L_{\bm{f}} b,
\end{equation}
then there is no feasible control action $u$ that can jointly satisfy the barrier function \eqref{eq:cbf-form} and control bounds \eqref{eq:control-bounds}.

\begin{assumption} \label{smp:monotonic}
    We only consider the dynamics of the system within the set $b(\bm{x})\leq 0$.
\end{assumption}

Assumption \ref{smp:monotonic} simplifies our analysis by assuming the system initially satisfies $b(\bm{x}) \leq 0$, and it restricts the form of our CBFs to class-$\mathcal{K}_{\infty}$ functions.
This can be relaxed by extending the optimal CBF to an extended class-$\mathcal{K}_{\infty}$ function, which has demonstrated robustness to noise, disturbances, and model uncertainty \cite{emam2022safe}.

\begin{assumption}\label{smp:signs}
    For a constraint $b(\bm{x})$ of order $r$, for every state $\bm{x}\in\mathbb{R}^m$, there exists at least one value of $\bm{u}$ satisfying
    \begin{align*}
        \min_{\bm{u}} \frac{d^r}{dt^r} b(\bm{x}) &\leq -\epsilon^2 < 0, \\
        \max_{\bm{u}} \frac{d^r}{dt^r} b(\bm{x}) &\geq 0,
    \end{align*}
    that also satisfies the control bounds \eqref{eq:control-bounds}.
\end{assumption}

We employ Assumption \ref{smp:signs} to simplify our analysis and ensure the system has the control authority to keep the system in the safe set.
This assumption is straightforward to relax with tools from Lyapunov stability, e.g., LaSalle's Invariance Principle \cite{KhalilBook}, to show forward invariance of $b(\bm{x})$.

\begin{assumption} \label{smp:differentiable}
    If the constraint $b(\bm{x})$ is order $r$, then it is also class $C^r$.
\end{assumption}

Assumption \ref{smp:differentiable} requires the constraint to be smooth enough for the control input $\bm{u}$ to appear after taking $r$ derivatives.
This is necessary because the control input is bounded, so the system trajectory cannot stay on a discontinuous constraint boundary.
This assumption is straightforward to ensure, as the form of the constraint $b(\bm{x})$ is generally chosen by a designer.

\section{First Order Systems}\label{sec:first-order}

We begin with a first order system where the control input $\bm{u}$ shows up explicitly in the first derivative of the barrier function.
We follow this with an analysis of second order systems in Section \ref{sec:second-order}, where the constraint compatibility problem is more significant.

We start by defining the set of \textit{safe states},
\begin{equation}
    \mathcal{C} = \big\{ \bm{x}\in\mathbb{R}^m ~:~ b(\bm{x}) \leq 0 \big\},
\end{equation}
which is compact under the continuity of $b(\bm{x})$ (Assumption \ref{smp:differentiable}).
Next, we seek a ZBF $\beta(\bm{x})$ to guarantee forward invariance of the set $\mathcal{C}$ via,
\begin{equation}
    \dot{b}(\bm{x}) \leq -\beta(\bm{x}).
\end{equation}
We require $\beta$ to satisfy,
\begin{equation} \label{eq:signs}
\sign{\big(-\beta(\bm{x})\big)} = 
    \begin{cases}
        1 &\text{ for all } x\in\mathcal{C}\setminus\partial\mathcal{C},\\
        0 &\text{ for all } x\in\partial\mathcal{C},
    \end{cases}
\end{equation}
where $\partial\mathcal{C}$ is the boundary of $\mathcal{C}$.
This satisfies the Nagumo theorem \cite{xiao2023safe}, which implies forward invariance of $\mathcal{C}$.

To guarantee constraint compatibility, $\beta$ must also satisfy,
\begin{equation}
\begin{aligned}
        -\beta(x) \geq \underset{\bm{u}}{\min}\big(\dot{b}(x)\big), \\
        -\beta(x) \leq \underset{\bm{u}}{\max}\big(\dot{b}(x)\big),
\end{aligned}
\end{equation}
which trivially guarantees constraint compatibility,
\begin{equation}
    \min_{\bm{u}} \big( \dot{b}(\bm{x}) \big) \leq \dot{b}(\bm{x}) \leq -\beta(\bm{x}) \leq \max_{\bm{u}}\big(\dot{b}(\bm{x})\big).
\end{equation}
To maximize the feasible set of control actions $\bm{\bm{u}}$, we select
\begin{equation} \label{eq:fo-betamax}
    -\beta(\bm{x}) = 
    \begin{cases}
        \underset{\bm{u}}{\max} \big( \dot{b}(\bm{x}) \big) \quad &\text{if } \bm{x}\in\interior{\mathcal{C}}, \\
        0 \quad &\text{if } \bm{x}\in\boundary{\mathcal{C}}, \\
    \end{cases}
\end{equation}
which satisfies \eqref{eq:signs} by Assumption \ref{smp:signs}.
For a control-affine system \eqref{eq:dynamics}, the optimal ZBF simplifies to,
\begin{equation} \label{eq:fo-opt-zbf}
    L_{\bm{g}}b\,\bm{u} \leq
        \begin{cases}
        \big|L_{\bm{g}} b| u_{\max} \quad &\text{if } x\in\interior{\mathcal{C}} \\
        -L_{\bm{f}} b &\text{if } x\in\boundary{\mathcal{C}},
    \end{cases}
\end{equation}
which is a linear in $\bm{u}$.

\begin{property} \label{prp:fo-obey}
    The ZBF \eqref{eq:fo-opt-zbf} obeys all state and control constraints.
\end{property}

\begin{proof}
    For all $\bm{x}\in\interior{\mathcal{C}}$, $||\bm{u}|| \leq u_{\max}$ by the definition of the dot product.
    For all $\bm{x}\in\boundary{\mathcal{C}}$, Assumption \ref{smp:signs} implies the existence of a $\bm{u}$ that simultaneously satisfies the ZBF and the control bounds.
\end{proof}

\begin{property} \label{prp:bounds}
    The ZBF \eqref{eq:fo-opt-zbf} maximizes the feasible action space of $\bm{u}$ for all $\bm{x}\in\mathcal{C}$.
\end{property}

Similar to Property \ref{prp:fo-obey}, this is implied by the definition of the dot product and the construction of \eqref{eq:fo-betamax}.

\begin{theorem} \label{thm:first-order}
    The optimal ZBF \eqref{eq:fo-opt-zbf} can be approximated arbitrarily well by jointly imposing a linear CBF and control bounds.
\end{theorem}

\begin{proof}
    Consider the control barrier function $\alpha\big(b(\bm{x})\big) = c_1 b(\bm{x})$ for some constant $c_1 > 0$.
    For all points $\bm{x}\in\boundary{\mathcal{C}}$, both functions imply $\dot{b}\big(\bm{x}\big) \leq 0$.
    Let $\bm{x}\in\interior{\mathcal{C}}$ such that $b(\bm{x}) \leq \epsilon$ for some small $\epsilon > 0$.
    Next, let
    \begin{equation}
        c_1 = \frac{1}{\epsilon} \max_{\bm{u}}\big\{\dot{b}(\bm{x})\big\}.
    \end{equation}
    This implies that
    \begin{equation}
        c_1 b(\bm{x}) \geq \max_u\big\{\dot{b}(\bm{x})\big\}\geq\dot{b}(\bm{x})
    \end{equation}
    for all $b(\bm{x}) \geq \epsilon$.
    Thus, the feasible space of joint linear CBF and control bound constraints can be brought $\epsilon$-close to the optimal ZBF.
\end{proof}

Theorem \ref{thm:first-order} requires the slope of a linear CBF to be arbitrarily steep to ensure $\epsilon$-closeness.
However, for any real system, the effect of any noise and extenral disturbances will quickly dominate the $\epsilon$-optimality gap for a linear CBF with a finite slope.
Thus, a linear CBF will effectively approximates the optimal ZBF in practical applications.

\section{Second Order ZBFs}\label{sec:second-order}

For higher-order systems, we define a sequence of nested safe sets \cite{Xiao2019ControlDegree}.
The largest safe set is
\begin{equation}
    \mathcal{C}_1 = \{ \bm{x} \in\mathbb{R}^m ~:~ b(\bm{x}) \leq 0\},
\end{equation}
which is identical to the first-order safe set in Section \ref{sec:first-order}.

To generate the higher-order safe set $\mathcal{C}_2$, we take advantage of the following lemma.
\begin{lemma} \label{lma:boundary}
    For any $\bm{x}\in\partial\mathcal{C}_1$, the set $\mathcal{C}_1$ is invariant if and only if $\dot{b}(\bm{x}) = 0$.
\end{lemma}

Lemma \ref{lma:boundary} is essentially a restatement of the Nagumo theorem.
This is equivalent to the tangency conditions in optimal control \cite{Bryson1975AppliedControl}, which arise when the system cannot apply infinite impulse to make instantaneous jumps in the state variables.

Next, we seek to define the set of all states $\bm{x}$ that guarantee recursive feasibility, i.e., states where there always exists a control input such that $\mathcal{C}_1$ is forward invariant.
We facilitate this with the definition of the \textit{shortest line integral}.

\begin{definition} \label{def:line-int}
The \textit{shortest line integral} starting at state $\bm{x}$ and ending at $\bm{x}_f$ satisfies,
\begin{equation}
    \min_{\bm{u}\in\mathcal{U}}\Big\{\int_{b(\bm{x})}^{b(\bm{x}_f)} \ddot{b}(\bm{x})\,db \Big\}= \int_{\dot{b}(\bm{x})}^{\dot{b}(\bm{x}_f)} \dot{b}(\bm{x})\,d\dot{b},
\end{equation}
which comes from the kinematic relationship between $b(\bm{x}), \dot{b}(\bm{x})$, and $\ddot{b}(\bm{x})$ \cite{dynamics}.
\end{definition}

Next, we prove that the dynamical motion primitive,
\begin{equation} \label{eq:dmp}
    \min_{\bm{u}} \Big\{ \ddot{b}(\bm{x}) \Big\},
\end{equation}
describes the system's trajectory on the safe set's boundary.

\begin{lemma} \label{lma:bangbang}
    The function $\ddot{b}$ that generates the shortest line integral (Definition \ref{def:line-int}) is the instantaneous maximum or minimum of $\ddot{b}$.
\end{lemma}

\begin{proof}
    We seek the function $\ddot{b}$ that minimizes,
    \begin{equation} \label{eq:parameterized-line}
        \int_{t_1}^{t_2} \ddot{b}(\bm{x}(t))|\dot{b}(\bm{x}(t))|dt,
    \end{equation}
    which is equivalent to the line integral in Definition \ref{def:line-int} for any parameterization of the trajectory that satisfies the boundary conditions \cite{dynamics}.
    This is an optimal control problem, and the unconstrained solution for $\ddot{b}$ satisfies \cite{beaver2024optimal},
    \begin{equation}
        -\frac{d}{dt}\Bigg(\ddot{b}\frac{|\dot{b}|}{\dot{b}}\Bigg) + \frac{d^2}{dt^2}\Big(|\dot{b}|\Big) = 0.
    \end{equation}
    Using the additive property of integrals, we split \eqref{eq:parameterized-line} into intervals where $\dot{b}$ has a positive or negative sign.
    In each of these cases, the optimality condition reduces to $0 = 0$.
    This implies that $\ddot{b}$ never follows the unconstrained trajectory, and the behavior of $\ddot{b}$ must be bang-bang.
\end{proof}

\begin{remark} \label{rmk:embedded}
    Definition \ref{def:line-int} and Lemma \ref{lma:bangbang} define a motion primitive $\min_{\bm{u}} \ddot{b}(\bm{x})$ over the horizon $b(\bm{x}) \in [b_1, 0]$.
    This horizon can be large, and if $b(\bm{x}, t)$ is time-varying, the motion primitive must also contain a model for the time-varying component $\frac{\partial^2 b}{\partial t^2}$.
    Thus, while CBFs are generally considered myopic, they actually contain an implicit model for i) the future trajectory of endogenous variables, and ii) the evolution of exogenous variables.
\end{remark}

\begin{lemma} \label{lma:c2-boundary}
    The boundary of the safe set $\partial\mathcal{C}_2$ is,
    \begin{align*}
        \partial\mathcal{C}_2 = \Big\{ \bm{x}\in R^m ~:~ \dot{b}(\bm{x}) \geq 0, \int_{b(\bm{x})}^{\bm{0}} \underline{\ddot{b}}\,db +  \frac{1}{2}\dot{b}(\bm{x})^2 = 0,
        \Big\} \\
        \bigcup\Big\{ \bm{x}\in R^m ~:~ \dot{b}(\bm{x}) < 0, \int_{b(\bm{x})}^{\bm{0}} \underline{\ddot{b}}\, db -  \frac{1}{2}\dot{b}(\bm{x})^2 = 0,
        \Big\},
    \end{align*}
    where $\underline{\ddot{b}}$ is the lower bound of $\ddot{b}$ (Lemma \ref{lma:bangbang}).
\end{lemma}

\begin{proof}
    Lemma \ref{lma:c2-boundary} evaluates the shortest line integral (Definition \ref{def:line-int}) from every $\bm{x}\in\mathbb{R}^m$ to the set boundary, which is necessary and sufficient for safety by Lemma \ref{lma:boundary}.
    
    For any other point $\bm{x}'$, let $\delta$ be the euclidean distance between $\bm{x}'$ and the nearest point $\bm{x}$ satisfying Lemma \ref{lma:c2-boundary}.
    No $\bm{y}\in\mathcal{B}_{\delta/2}(\bm{X}')$ satisfies Lemma \ref{lma:c2-boundary}, where $\mathcal{B}_r(\bm{x})$ is a ball of radius $r$ centered at $\bm{x}$.
    Thus, $\mathcal{B}_{\delta/2}(\bm{x}')$ is an open set (of safe or unsafe states), and $\bm{x}'\not\in\partial\mathcal{C}_2$ by definition.
\end{proof}

Note that the definition of $\partial\mathcal{C}_2$ contains two cases, one for each sign of $\dot{b}$.
This is because when $\dot{b}(\bm{x}) < 0$, the limits of integration change direction in Definition \ref{def:line-int}.
Intuitively, $\dot{b}(\bm{x}) < 0$ implies that $b(\bm{x})$ is decreasing--thus, by the continuty of $b(\bm{x})$ (Assumption \ref{smp:differentiable}), there exists an interval of time $\Delta t$ such that $b(\bm{x}) \leq 0$ for any arbitrarily large control input $\bm{u}$.

\begin{lemma}\label{lma:c2}
    The safe set $\mathcal{C}_2$ is
        \begin{align*}
        \mathcal{C}_2 =& \Big\{ \bm{x}\in R^m ~:~ \dot{b}(\bm{x}) \geq 0, \int_{b(x)}^{\bm{0}} \underline{\ddot{b}}\,db +  \frac{1}{2}\dot{b}(\bm{x})^2 \leq 0,
        \Big\} \\
        &\bigcup\Big\{ \bm{x}\in \mathbb{R}^m ~:~ \dot{b}(\bm{x}) < 0, \bm{x}\in {C}_1
        \Big\}.
    \end{align*}
\end{lemma}

\begin{proof}
    Let $\bm{x}_f\in\mathcal{C}_1$ such that $\dot{b}(\bm{x}_f) = 0$.
    We use the contrapositive of Lemma \ref{lma:boundary}, which implies that $\bm{x}_f$ is feasible if and only if $ b(\bm{x}_f) \leq 0$.

    First, consider some $\bm{x}\in\mathcal{C}_2$ such that  $\dot{b}(\bm{x}) \geq 0$.
    Evaluating the shortest line integral (Definition \ref{def:line-int}) yields,
    \begin{equation}
        \int_{b(\bm{x})}^{b(\bm{x}_f)} \underline{\dot{b}}\,db + \frac{1}{2}\dot{b}(\bm{x})^2 = 0.
    \end{equation}
    This implies,
    \begin{equation}
        \int_{b(\bm{x})}^{0} \underline{\dot{b}}\,db + \frac{1}{2}\dot{b}(\bm{x})^2 = \int_{b(\bm{x}_f)}^{0} \underline{\ddot{b}}\,db \leq 0.
    \end{equation}
    Assumption \ref{smp:signs} implies that the integrand of the right hand side is negative and $b$ is increasing from $b(\bm{x}_f)\leq 0$ to $b = 0$.
    Thus, the condition in Lemma \ref{lma:c2} is necessary for safety.

    Next, consider a point $\bm{x}\not\in\mathcal{C}_2$ such that $\dot{b}(\bm{x}) \geq 0$.
    Similar to the previous case, applying Definition \ref{def:line-int} yields
    \begin{equation}
        \int_{b(\bm{x})}^{0} \underline{\dot{b}}\,db + \frac{1}{2}\dot{b}(\bm{x}) = \int_{b(\bm{x}_f)}^{0} \underline{\ddot{b}}\,db > 0.
    \end{equation}
    Under our premise, the integrand of the right hand side is negative and $b$ is decreasing from $b(x_f) > 0$ to $b = 0$ under Assumption \ref{smp:signs}.
    Thus, the premise of Lemma \ref{lma:c2} is sufficient for safety by contraposition.

    Finally, consider the case where $\bm{x}\in\mathcal{C}_1$ such that $\dot{b}(\bm{x}) < 0$.
    If $\dot{b}$ remains non-positive, then $b\leq \bm{b}(x) \leq 0$ and the trajectory remains in $\mathcal{C}_1$.
    If $\dot{b}$ becomes positive, then by the continuity in $\dot{b}$ (Assumption \ref{smp:differentiable}) it must pass through some intermediate point $\dot{b}(\bm{x}') = 0$.
    Setting $\bm{x}_f = \bm{x}'$ satisfies the premise of the previous case and completes the proof.
\end{proof}

\begin{lemma} \label{lma:subset}
    The set $\mathcal{C}_2$ is a strict subset of $\mathcal{C}_1$.
\end{lemma}

\begin{proof}
    
    For all $x\in\mathbb{R}^m$ such that $\dot{b}(\bm{x}) < 0$, $\bm{x}\in\mathcal{C}_1$ if and only if $\bm{x} \in \mathcal{C}_2$.
    
    Next, let $x\in\mathcal{C}_2$ such that $\dot{b}(\bm{x}) \geq 0$.
    By Lemma \ref{lma:c2},
    \begin{equation} \label{eq:subset-inequality}
        \int_{b(\bm{x})}^{0} \underline{\ddot{b}}(\bm{x})\,db \leq -\frac{1}{2}\dot{b}^2(\bm{x}) \leq 0.
    \end{equation}
    Under Assumption \ref{smp:signs} the integrand of \eqref{eq:subset-inequality} is negative, therefore $b(\bm{x}) \leq 0$ and $\bm{x}\in\mathcal{C}_1$.
    
    To prove strictness of the subset, consider a point $\bm{x}\in\mathcal{C}_1$ such that $b(\bm{x}) = 0$ and $\dot{b}(\bm{x}) > 0$.
    By Lemma \ref{lma:boundary}, $\bm{x}\not\in\mathcal{C}_2$.
\end{proof}

\begin{theorem}\label{thm:second-order}
    The largest set that ensures forward invariance of $b(\bm{x}) \leq 0$ is,
        \begin{align*}
        \mathcal{C}_2 =& \Big\{ \bm{x}\in R^m ~:~ \dot{b}(\bm{x}) \geq 0, \int_{b(x)}^{\bm{0}} \underline{\ddot{b}}\,db +  \frac{1}{2}\dot{b}(\bm{x})^2 \leq 0,
        \Big\} \\
        &\bigcup\Big\{ \bm{x}\in \mathbb{R}^m ~:~ \dot{b}(\bm{x}) < 0, \bm{x}\in {C}_1
        \Big\}.
    \end{align*}
\end{theorem}

\begin{proof}
    The proof follows from Lemmas \ref{lma:c2} and \ref{lma:subset}.
\end{proof}

Theorem \ref{thm:second-order} describes the shape of the largest safe set that ensures recursive feasibility of $\mathcal{C}_1$.
Next, we present a result that we use to define the optimal CBF, which has the same feasible action space as $\mathcal{C}_2$.

\begin{lemma} \label{lma:class-k}
    The shortest line integral (Definition \ref{def:line-int}) induces a Class $\mathcal{K}_{\infty}$ function.
\end{lemma}

\begin{proof}
    Consider the function,
    \begin{equation} \label{eq:fbx}
        F\big(-b(\bm{x})\big) =  -\int_{b(\bm{x})}^{0}\underline{\ddot{b}} db.
    \end{equation}
    This satisfies the definition of a class $\mathcal{K}_{\infty}$ function: 1) $F(0) = 0$, as the integral over a zero-measure set is zero; 2) $F(b(\bm{x})$ is strictly increasing, and 3) $F(b(\bm{x}))\to\infty$ as $b(\bm{x}) \to \infty$, since $\underline{\ddot{b}}(\bm{x})$ is upper bounded by a negative constant.
\end{proof}

\begin{theorem} \label{thm:cbf-optimal}
    The control barrier function
    \begin{equation} \label{eq:opt-cbf}
        \alpha\big(b(\bm{x})\big) = \sqrt{-2\,\int_{b(\bm{x})}^{0} \underline{\ddot{b}}\, db}
    \end{equation}
    combined with a control bounds constraint is optimal, i.e., its set of safe actions is equal to $\mathcal{C}_2$.
\end{theorem}

\begin{proof}
    By Lemma \ref{lma:class-k}, \eqref{eq:opt-cbf} is a class $\mathcal{K}_{\infty}$ function, and it describes the safe set,
    \begin{equation}
        \mathcal{C}_{\alpha} = \big\{ \bm{x}\in\mathcal\mathbb{R}^m ~:~ \dot{b}(\bm{x}) \leq -\alpha\big(b(\bm{x})\big) \big\}.
    \end{equation}
    Substituting \eqref{eq:fbx} yields the boundary,
    \begin{equation}
        \dot{b}(\bm{x}) \leq - \sqrt{-2\int_{b(\bm{x})}^{0}\underline{\dot{b}}}.
    \end{equation}
    Squaring both sides and re-arranging yields the boundary of $\mathcal{C}_2$, thus \eqref{eq:opt-cbf} describes the optimal barrier function when $\dot{b} \leq 0$.
    Enforcing the control bounds when $\dot{b}(\bm{x}) > 0$ completes the proof.
\end{proof}

Theorem \ref{thm:cbf-optimal} is a constructive proof to generate the control barrier function that exactly captures the system's feasible action space.
Keeping the system in $\mathcal{C}_{\alpha}$ is a first-order constraint, and thus applying Theorem \ref{thm:first-order} yields a control barrier function that is $\epsilon-$close to the optimal ZBF.
Finally, to complete this section, we describe the optimal behavior of a system with second-order constraints under the optimal control barrier function.

\begin{remark} \label{rmk:switching}
    The optimal control policy switches between satisfying the control bounds when $\bm{x}\in\interior{\mathcal{C}}_{\alpha}$ and minimizing $\ddot{b}(\bm{x})$ when $\bm{x}\in\boundary{\mathcal{C}}_{\alpha}$; this is depicted in Fig. \ref{fig:switching-system}.
\end{remark}

\begin{figure}[ht]
    \centering
    \begin{tikzpicture}
    \tikzstyle{block}=[draw,,minimum width=3cm];
    \tikzstyle{arrow}=[->,ultra thick];
    \node[block] (A) at (0, 0) {$||\bm{u}|| \leq u_{\max}$};
    \node[block] (B) at (0, -1.5) {$\bm{u} = \arg\min_{\bm{u}}\{\ddot{b}(\bm{x})\}$};
    \draw[arrow] (A.300) -- node[right]{$\bm{x}\in\boundary{\mathcal{C}}_2$} (B.60);
    \draw[arrow] (B.120) -- node[left]{$\bm{x}\not\in\boundary{\mathcal{C}}_2$} (A.240);
    \end{tikzpicture}
    \caption{Switching description of the optimal behavior; the system operates within the control bounds until it reaches the safe set boundary.}
    \label{fig:switching-system}
\end{figure}
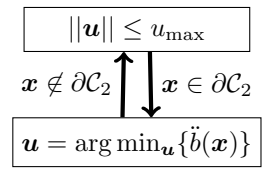

\begin{proof}
We now have the first order constraint,
\begin{equation}
    \dot{b}(\bm{x}) + \alpha(b(\bm{x})) \leq 0.
\end{equation}
Applying Theorem \ref{thm:second-order} implies that the optimal CBF approximates the ZBF,
\begin{equation} \label{eq:cbf-2o}
    -\beta(\bm{x}) = 
    \begin{cases}
        \max_{\bm{u}}\{ \ddot{b}(\bm{x}) + \dot{\alpha}(b(\bm{x})) \} &\text{ if } \bm{x}\in\interior{\mathcal{C}}_2, \\
        0 &\text{ if } \bm{x}\in\boundary{\mathcal{C}}_2.
    \end{cases}
\end{equation}
The ZBF \eqref{eq:cbf-2o} is equivalent to enforcing the control bounds for $\bm{x}\in\interior{\mathcal{C}}_2$ by definition.
For $\bm{x}\in\boundary{\mathcal{C}}_2$, the ZBF keeps the system within $\mathcal{C}_2$.
By construction, this is only be achieved by applying the control input 
\begin{equation}
    \bm{u} = \arg\min_{\bm{u}}\{ \ddot{b}(\bm{x})\}.
\end{equation}
\end{proof}

\section{Simulation Results} \label{sec:simulation}

In this section we consider longitudinal control of a connected autonomous vehicle (CAV) using adaptive cruise control.
Under adaptive cruise control, the CAV attempts to maintain a set speed while maintaining a safe following distance.
At each time-instant the vehicle measures the distance $\delta(t)$ to the vehicle in front of it, which must be kept within an admissible range.

The autonomous vehicle has a state $\bm{x} = [p, v]^{\intercal}\in\mathbb{R}^2$ and control action $u\in\mathbb{R}$.
It follows integrator dynamics,
\begin{equation} \label{eq:example-dynamics}
    \dot{p} = v, \quad \dot{v} = u.
\end{equation}
We impose the acceleration bounds
\begin{equation} \label{eq:example-bounds}
    |u| \leq u_{\max},
\end{equation}
and we constrain the vehicle to obey a headway constraint,
\begin{equation} \label{eq:example-constraint}
   b(\bm{x}) = p - \delta(t) - \gamma \leq 0,
\end{equation}
where $\delta(t)$ is the position of the vehicle in front of the CAV, and $\gamma$ is a minimum separating distance.

Remark \ref{rmk:embedded} requires us to explicitly embed a model of $\delta(t)$ in the optimal CBF.
For vehicle in front, we assume a constant speed model, i.e., $\ddot{\delta}(t) = 0$.
This is a common model used in the literature with reasonable success (e.g., for model predictive control \cite{mahbub2022platoon}), however, in safety-critical scenarios it may be beneficial to replace $\ddot{\delta}(t)$ with a lower bound (e.g., the minimum braking distance of \cite{beaver2021constraint}).
Alternatively, future predictions of $\delta(t)$ can be communicated from the lead vehicle if it is a CAV.

To construct the higher order CBF, we consider the case when $v\geq 0$.
From Lemma \ref{lma:bangbang}, the shortest line integral is,
\begin{equation}
    \int_{b=p-\delta}^{b=0} -u_{\max}\, dp = u_{\max}(p-\delta(t)-\gamma).
\end{equation}
Theorem \ref{thm:cbf-optimal} implies the barrier function
\begin{equation}
    \alpha\big(b(\bm{x})\big) = -\sqrt{-2\,u_{\max}\,(p-\delta(t)-\gamma)},
\end{equation}
and thus,
\begin{equation} \label{eq:ex-v-cbf}
    v - \dot{\delta}(t) \leq \sqrt{-2\,u_{\max}\,(p-\delta(t)-\gamma)}.
\end{equation}
Equivalently, the boundary of the safe set (Theorem \ref{thm:second-order}) yields the minimum stopping distance constraint,
\begin{equation}
    2\,u_{\max}(p-\delta(t)-\gamma) + \big(v-\dot{\delta}(t)\big)^2 \leq 0.
\end{equation}

Following , Remark \ref{rmk:switching} the optimal barrier function to ensure \eqref{eq:ex-v-cbf} is,
\begin{align} \label{eq:example-final-cbf}
    u \leq& -c_1 \Big(v - \dot{\delta}(t) - \sqrt{-2 u_{\max}(p-\delta(t)-\gamma)}\Big) \notag\\
    &- \frac{u_{\max} (v - \dot{\delta}(t))}{\sqrt{-2 u_{\max}(p-\delta(t)-\gamma)}},
\end{align}
for some large constant $c_1$.
Finally, we present a quadratic program to solve the adaptive cruise control problem for the CAV.

\begin{problem}{(Adaptive Cruise Control)} \label{prb:ocp} At each time step $t_k$, apply the control action $u$ that maintains the optimal speed $v^*$ while guaranteeing safety:
    \begin{align*}
        \min_{u(t_k)} & \big(v(t_k) + u(t_k)\Delta t - v^*\big)^2 \\
        \text{subject to:}& \\
        \eqref{eq:example-dynamics}& \quad \text{(dynamics)}, \\
        \eqref{eq:example-bounds}& \quad \text{(control bounds)}, \\
        \eqref{eq:example-final-cbf}& \quad \text{(control barrier function; if $v > \delta(t)$)}.
    \end{align*}
    Note that Theorem \ref{thm:second-order} implies we only need to enforce the CBF when $\dot{b}(\bm{x}) = v > 0$.
\end{problem}

It is critical to realize that while the CBF's safety guarantees are derived for a continuous-time system, we discretize Problem \ref{prb:ocp} for implementation on a digital computer.
In general, we calculate the right hand side of the CBF \eqref{eq:example-final-cbf} using the state at the current time $t_k$, and we apply the constant control input $u(t_k)$ over the interval $\Delta t$. 
The impact of discretizing the system, and managing constraint violations over the interval $\Delta t$, is an area of active research \cite{xiao2022event}.

A simulation of Problem \ref{prb:ocp} is depicted in Fig. \ref{fig:sim-results1} for the optimal CBF and a second-order linear CBF with the form,
\begin{equation} \label{eq:linear-cbf}
    u \leq -c_B\big( v + \dot{\delta}(t) \big) - c_A c_B\big( p + \delta(t) + \gamma \big).
\end{equation}
The corresponding parameters for each simulation are given in Table \ref{tab:parameters}.

\begin{table}[ht]
    \centering
    \begin{tabular}{cccccccccc}
        $\bm{x}^0$ & $v^*$ & $\gamma$ &$\delta$ & $\dot{\delta}$ & $u_{\max}$ & $c_1$ & $c_a$ & $c_b$ \\ \toprule
        $[0,\, 10]^{\intercal}$ & $10$ & $10$ & $1$ & $10$ & $5$ & $3$ & $100$ & $1$
    \end{tabular}
    \caption{Parameters used for the adaptive cruise control simulation.}
    \label{tab:parameters}
\end{table}

Fig. \ref{fig:sim-results1} depicts how the upper bound on the control input $u$ varies as a function of the rear-end safety constraint $b(\bm{x})$.
As expected, the linear CBF is conservative, and causes the system to start braking before it is necessary.
This is not an artefact of parameter selection, rather it comes from the linear CBF \eqref{eq:linear-cbf} under-approximating the square root in the optimal CBF \eqref{eq:example-final-cbf}.
In contrast, the optimal CBF approximates the switching behavior described in Remark \ref{rmk:switching}.
While the conservativeness of the linear CBF may have additional benefits, e.g., passenger comfort, such considerations should be explicitly imposed on the controller in Problem \ref{prb:ocp}--they should not arise as coincidence of selecting a sub-optimal CBF.

\begin{figure}[ht]
    \centering
    \includegraphics[width=0.8\linewidth]{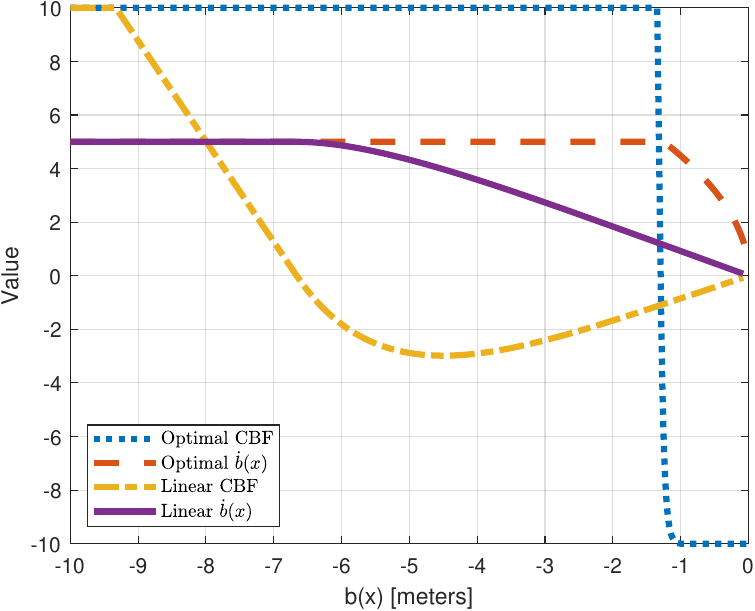}
    \caption{Behavior of the optimal CBF (blue and orange) and the conservative linear CBF (purple and yellow) as the system approaches the boundary of the safe set at $b(\bm{x}) =0$.}
    \label{fig:sim-results1}
\end{figure}

Both the optimal and linear CBFs approach a steady separating distance equal to $\gamma$, and in each case the relative speed between the vehicles is zero.
This is necessary and sufficient for forward invariance of the safe set (Lemma \ref{lma:boundary}).
It is important to note that the CAV is only guaranteed to satisfy the rear-end safety constraint as long as our modeling assumption of $\ddot{\delta}(t) = 0$ holds.
If the front vehicle applies a positive acceleration, the proposed optimal CBF would be conserviative with respect to the feasible action space.
Similarly, if the front vehicle decelerates, our CBF would no longer guarantee forward-invariance for the safe set. 
Overcoming these challenges requires an explicit model of the front vehicle--either to predict $\ddot{\delta}(t)$ or to to model the worst case scenario where $\ddot{\delta}(t)$ is equal to its lower bound (emergency braking).
However, this is an issue for problem formulation and modeling that is independent of the specific barrier function.

\section{Conclusions} \label{sec:conclusion}

In this article we presented sufficient conditions and a constructive proof for control barrier functions to be optimal, i.e., they maximize the feasible action space of a system while guaranteeing safety.
We proved that the optimal CBF follows a dynamical motion primitive on the boundary of the safe set (Definition \ref{def:line-int}) and implicitly predicts the future trajectories of exogenous variables (Remark \ref{rmk:embedded}).
We demonstrated the performance of the optimal CBF in an adaptive cruise control problem and compared its performance to a linear CBF.

This work suggests several interesting areas for future work.
The most immediate is generalizing the results for higher-order systems, as well as specializing the results to systems with scalar dynamics.
Finding methods to solve the shortest path integral is an interesting area of research, and it presents an opportunity for data-driven and machine learning solutions to generate optimal barrier functions.
Addressing the constraint compatability for a system with multiple optimal CBFs is another area interest.
Finally, analyzing the effect of the stopping behavior and exogenous variables that are embedded in the optimal CBF through the \textit{shortest line integral} is a compelling research direction for large multi-agent systems.

\addtolength{\textheight}{-12cm}   

\bibliographystyle{unsrt}
\bibliography{mendeley,cbfs}

\begin{thebibliography}{10}

\bibitem{Ames2019ControlApplications}
Aaron~D. Ames, Samuel Coogan, Magnus Egerstedt, Gennaro Notomista, Koushil Sreenath, and Paulo Tabuada.
\newblock {Control barrier functions: Theory and applications}.
\newblock In {\em 2019 18th European Control Conference, ECC 2019}, pages 3420--3431. Institute of Electrical and Electronics Engineers Inc., 6 2019.

\bibitem{xiao2023safe}
W.~Xiao, C.G. Cassandras, and C.~Belta.
\newblock {\em Safe Autonomy with Control Barrier Functions: Theory and Applications}.
\newblock Synthesis Lectures on Computer Science. Springer International Publishing, 2023.

\bibitem{Xiao2021BridgingVehicles}
Wei Xiao, Christos~G. Cassandras, and Calin~A. Belta.
\newblock {Bridging the Gap between Optimal Trajectory Planning and Safety-Critical Control with Applications to Autonomous Vehicles}.
\newblock {\em Automatica}, 129, 8 2021.

\bibitem{Egerstedt2018RobotAutonomy}
Magnus Egerstedt, Jonathan~N. Pauli, Gennaro Notomista, and Seth Hutchinson.
\newblock {Robot ecology: Constraint-based control design for long duration autonomy}.
\newblock {\em Annual Reviews in Control}, 46:1--7, 1 2018.

\bibitem{Notomista2019TheRobot}
Gennaro Notomista, Yousef Emam, and Magnus Egerstedt.
\newblock {The SlothBot: A Novel Design for a Wire-Traversing Robot}.
\newblock {\em IEEE Robotics and Automation Letters}, 4(2):1993--1998, 4 2019.

\bibitem{grandia2021multi}
Ruben Grandia, Andrew~J Taylor, Aaron~D Ames, and Marco Hutter.
\newblock Multi-layered safety for legged robots via control barrier functions and model predictive control.
\newblock In {\em 2021 IEEE International Conference on Robotics and Automation (ICRA)}, pages 8352--8358. IEEE, 2021.

\bibitem{singletary2022onboard}
Andrew Singletary, Aiden Swann, Yuxiao Chen, and Aaron~D Ames.
\newblock Onboard safety guarantees for racing drones: High-speed geofencing with control barrier functions.
\newblock {\em IEEE Robotics and Automation Letters}, 7(2):2897--2904, 2022.

\bibitem{Xu2022FeasibilityFunctions}
Kaiyuan Xu, Wei Xiao, and Christos~G. Cassandras.
\newblock {Feasibility Guaranteed Traffic Merging Control Using Control Barrier Functions}.
\newblock In {\em 2022 American Control Conference}, pages 2039--2314, 2022.

\bibitem{Tan20222CompatibilitySystems}
Xiao Tan and Dimos~V. Dimarogonas.
\newblock Compatibility checking of multiple control barrier functions for input constrained systems.
\newblock In {\em 2022 IEEE 61st Conference on Decision and Control (CDC)}, pages 939--944, 2022.

\bibitem{Xiao2023LearningCBFs}
Wei Xiao, Christos~G. Cassandras, and Calin~A. Belta.
\newblock Learning feasibility constraints for control barrier functions.
\newblock In {\em 2023 European Control Conference (ECC)}, pages 1--6, 2023.

\bibitem{lindemann2018control}
Lars Lindemann and Dimos~V Dimarogonas.
\newblock Control barrier functions for signal temporal logic tasks.
\newblock {\em IEEE control systems letters}, 3(1):96--101, 2018.

\bibitem{notomista2019optimal}
Gennaro Notomista, Siddharth Mayya, Seth Hutchinson, and Magnus Egerstedt.
\newblock An optimal task allocation strategy for heterogeneous multi-robot systems.
\newblock In {\em 2019 18th European Control Conference (ECC)}, pages 2071--2076. IEEE, 2019.

\bibitem{emam2022safe}
Yousef Emam, Gennaro Notomista, Paul Glotfelter, Zsolt Kira, and Magnus Egerstedt.
\newblock Safe reinforcement learning using robust control barrier functions.
\newblock {\em IEEE Robotics and Automation Letters}, pages 1--8, 2022.

\bibitem{KhalilBook}
Hassan~K Khalil.
\newblock {\em Nonlinear Systems}.
\newblock Prentice Hall, 2002.

\bibitem{Xiao2019ControlDegree}
Wei Xiao and Calin Belta.
\newblock {Control Barrier Functions for Systems with High Relative Degree}.
\newblock In {\em Proceedings of the IEEE Conference on Decision and Control}, volume 2019-December, pages 474--479. Institute of Electrical and Electronics Engineers Inc., 12 2019.

\bibitem{Bryson1975AppliedControl}
A.~E. Bryson, Jr. and Yu-Chi Ho.
\newblock {\em {Applied Optimal Control: Optimization, Estimation, and Control}}.
\newblock John Wiley and Sons, 1975.

\bibitem{dynamics}
James~L Meriam, L~Glenn Kraige, and Jeff~N Bolton.
\newblock {\em Engineering mechanics: dynamics}.
\newblock John Wiley \& Sons, 2020.

\bibitem{beaver2024optimal}
Logan~E Beaver and Andreas~A Malikopoulos.
\newblock Optimal control of differentially flat systems is surprisingly easy.
\newblock {\em Automatica}, 159:111404, 2024.

\bibitem{mahbub2022platoon}
AM~Ishtiaque Mahbub and Andreas~A Malikopoulos.
\newblock Platoon formation in a mixed traffic environment: A model-agnostic optimal control approach.
\newblock In {\em 2022 American Control Conference (ACC)}, pages 4746--4751. IEEE, 2022.

\bibitem{beaver2021constraint}
Logan~E Beaver et~al.
\newblock Constraint-driven optimal control of multiagent systems: A highway platooning case study.
\newblock {\em IEEE Control Systems Letters}, 6:1754--1759, 2021.

\bibitem{xiao2022event}
Wei Xiao, Calin Belta, and Christos~G Cassandras.
\newblock Event-triggered control for safety-critical systems with unknown dynamics.
\newblock {\em IEEE Transactions on Automatic Control}, 2022.

\end{thebibliography}

\end{document}